\documentclass[12pt]{article}
\usepackage{amssymb,latexsym,amsmath}
\usepackage{epsfig}
\newcommand{\fin}{\,\rule{1ex}{1.6ex}\,}

\title{\hspace*{1.25cm}On Gautschi's conjecture for generalized Gauss-Radau and Gauss-Lobatto formulae}
\author{H\'edi Joulak
\thanks{Laboratoire Paul Painlev{\'e}, UMR CNRS 8524, UFR de
Math\'ematiques, Universit\'e des Sciences et Technologies de
Lille, 59655 Villeneuve d'Ascq cedex -- France, e-mail: {\tt
hedi.joulak@gmail.com}} \& Bernhard Beckermann\thanks{Laboratoire
Paul Painlev{\'e}, UMR CNRS 8524, UFR de Math\'ematiques,
Universit\'e des Sciences et Technologies de Lille, 59655
Villeneuve d'Ascq cedex -- France, e-mail: {\tt
bbecker@math.univ-lille1.fr}} }
\date{~}
 \newtheorem{theorem}{Theorem}

 \newtheorem{remark}[theorem]{Remark}
 
 \newtheorem{corollary}[theorem]{Corollary}
 \newtheorem{lemme}[theorem]{Lemma}
\newenvironment{proof}%
{\rm \trivlist \item[\hskip \labelsep{\bf Proof. }]}%
{\hspace*{\fill}$\Box$\endtrivlist}

\def\nset{\hbox{l\hskip-2ptN}}

\def\pset{\hbox{I\hskip-2ptP}}

\begin{document}

\maketitle

\begin{abstract}
    Recently, Gautschi introduced so-called
    generalized Gauss-Radau and Gauss-Lobatto formulae which
    are quadrature formulae of Gaussian type involving not only
    the values but also the
    derivatives of the function at the endpoints. In the present note we show
    the positivity of the
    corresponding weights; this positivity has been conjectured already by Gautschi.

    As a consequence, we establish several convergence theorems
    for these quadrature formulae.
\end{abstract}

\noindent {\bf Keywords.} Quadrature formula, Gauss-Radau formula,
Gauss-Lobatto formula, orthogonality with varying weights.

\noindent {\bf AMS subject classifications.} 65D30, 42C05.


\section{Introduction}

In a recent paper \cite{Gau2}, Gautschi considered so-called
generalized Gauss-Radau and Gauss-Lobatto formulae which are
quadrature formulae of Gaussian type, i.e., having a highest
possible degree of exactness, and involving not only the values
but also the derivatives of the function at the endpoints of the
interval of integration. Such formulae are of the form
\begin{eqnarray} \nonumber
       \int f(t) \, d\lambda(t) &=&
      \sum_{j=0}^{r-1} \lambda_0^{(j)} f^{(j)}(a)
      + \sum_{j=1}^n \lambda_j f(\tau_j) +
      \sum_{j=0}^{s-1} (-1)^j \lambda_{n+1}^{(j)} f^{(j)}(b) +
      R_{n,r,s}(f)
      \\&=:& \label{gauss_radau_lobatto}
      Q_{n,r,s}(f) + R_{n,r,s}(f),
\end{eqnarray}
where $\lambda$ is a positive measure with support being a subset
of $[a,b]$ having an infinite number of points of increase in
$(a,b)$, and the integers $r,s\geq 0$ are the multiplicities of
the endpoints $a$ and $b$, respectively. In what follows we will
also allow the case $a=-\infty$ (or $b=+\infty$) of a possibly
unbounded support in which case only $r=0$ (or $s=0$,
respectively) is considered, that is, the corresponding sum in
$Q_{n,r,s}(f)$ does vanish.

It is well-known and easily verified that our requirement of
highest possible degree of exactness leads to a unique quadrature
formula of the form (\ref{gauss_radau_lobatto}) with degree of
exactness being equal to $2n+r+s-1$ \begin{equation}
\label{exactness}
      \forall f \in \mathbb P_{2n+r+s-1} : \qquad R_{n,r,s}(f)=0 ,
\end{equation}
where here and in what follows $\mathbb P_m$ denotes the space of
real polynomials of degree at most $m$. Here the free nodes
$\tau_j$ have to be chosen as to be the simple zeros of the $n$th
orthogonal polynomial with respect to the modified measure
$(x-a)^r (b-x)^s d\alpha(x)$ on $(a,b)$ being clearly positive,
and hence $\tau_j \in (a,b)$. Indeed, we get for $Q_{n,0,0}$ the
classical Gaussian quadrature rule, for $Q_{n,1,0}$ and
$Q_{n,0,1}$ the Gauss-Radau formulae and for $Q_{n,1,1}$ the
Gauss-Lobatto formula. For all these classical quadrature formulae
it is known that the weights $\lambda_1,...,\lambda_n$ and
possibly $\lambda_{0}^{(0)},\lambda_{n+1}^{(0)}$ are strictly
positive. The generalized Gauss-Radau (Gauss-Lobatto) formulae of
\cite{Gau2} are obtained for $s=0$ (and $r=s$, respectively).

Based on extensive numerical experiments for Jacobi, Laguerre and
elliptic Chebyshev measures using the numerical tools and methods
described in \cite{Gau1}, Gautschi conjectured in
\cite[Section~2.2 and Section~3.2]{Gau2} that the weights of the
generalized Gauss-Radau and Gauss-Lobatto formulae are all
strictly positive. He proved himself this conjecture for the inner
weights $\lambda_1,...,\lambda_n$ as well as for some boundary
weights, namely $\lambda_0^{(r-1)},\lambda_0^{(r-2)}>0$ for the
generalized Gauss-Radau formulae $Q_{n,r,0}$, and
$\lambda_0^{(r-1)},\lambda_0^{(r-2)},\lambda_{n+1}^{(r-1)},\lambda_{n+1}^{(r-2)}>0$
for the generalized Gauss-Lobatto formulae  $Q_{n,r,r}$. However,
the sign of the other weights remained an open question.

The aim of this paper is to show in Theorem~\ref{main} below that
Gautschi's conjecture is true, namely, all weights in the
quadrature formulae $Q_{n,r,s}$ are strictly positive. For this we
will show the slightly stronger result that suitable underlying
Lagrange polynomials (in the Hermite sense) do not change sign in
$(a,b)$. As a consequence, we obtain in
Corollary~\ref{banach_steinhaus} convergence of the quadrature
$Q_{n,r,s}(f)$ for fixed $r,s$ and $n\to \infty$ for sufficiently
differentiable functions $f$. The case of $r,s,n \to \infty$ is
discussed in Theorem~\ref{analytic} where we establish a geometric
rate of convergence for analytic $f$.

Before stating and proving our results in the next sections, we
should mention that generalized Gauss-Radau and Gauss-Lobatto
formulae are of major interest in different applications, and in
particular in moment preserving spline approximation on a compact
interval $[a,b]=[0,1]$, see \cite{FGM}, \cite[\S~3.3]{Gau1}, and
\cite[\S~4]{Gau2}: given a function $f\in \mathcal C^{m+1}([0,1])$
with moments $\mu_j=\int_0^1 t^j f(t)\, dt$, we are looking for a
partition $0=\tau_0<\tau_1<...<\tau_n<\tau_{n+1}=1$ and a spline
$\sigma$ of class $\mathcal C^{m-1}$ being piecewise $\mathbb P_m$
on each $[\tau_j,\tau_{j+1}]$ for $j=0,1,...,n$ and having the
same moments
\begin{equation} \label{spline_moments}
     \forall j=0,1,...,N : \qquad \int_0^1 t^j \sigma(t) \, dt =
     \mu_j ,
\end{equation}
with $N$ as large as possible, or in other words, the error
$f-\sigma$ is orthogonal to $\mathbb P_N$ with respect to Lebesgue
measure. By \cite[Theorem~3.61]{Gau1}, such a spline $\sigma$
exists for $N=2n+m$ if and only if the measure
$d\lambda(t)=(-1)^{m+1}f^{(m+1)}(t)/(m!) \, dt$ on $[0,1]$ has a
generalized Gauss-Radau quadrature formula $Q_{n,m+1,m+1}$ as in
(\ref{gauss_radau_lobatto}), and in this case the spline is given
by the quadrature data via $$
    \sigma(t)= \sum_{j=1}^n \lambda_j (\tau_j-t)_+^m  \, + \,
    \sum_{j=0}^m \frac{(t-1)^j}{j!} [ f^{(j)}(1)+(-1)^m \, m! \,
    \lambda_{n+1}^{(m-j)}] .
$$

\section{Positivity of the weights}

\begin{theorem}\label{main}
    All weights in the Gauss-type quadrature formula $Q_{n,r,s}$ given in
    (\ref{gauss_radau_lobatto}) are strictly positive for all integers
    $n,r,s\geq 0$
    \begin{eqnarray}
       \label{main1} j=1,2,...,n: && \lambda_j > 0 ,
       \\
       \label{main2} j=0,1,2,...,r-1: && \lambda_0^{(j)} > 0 ,
       \\
       \label{main3} j=0,1,2,...,s-1: && \lambda_{n+1}^{(j)} > 0 .
    \end{eqnarray}
\end{theorem}
\begin{proof}
The property (\ref{main1}) has already been
established by Gautschi \cite{Gau2}, for the sake of completeness
we repeat here the proof: consider $$
     p_j(t)=\frac{\omega_j(t)}{\omega_j(\tau_j)} , \quad
     \omega_j(t)=(t-a)^r (b-t)^s
     \prod_{k=1,k\neq j}^n (\tau_k-t)^2 ,
$$ then it is clear by construction that $p_j(\tau_j)=1$, and
$p_j\in \mathbb P_{2n+r+s-2}$ is non negative on $(a,b)$.
According to (\ref{exactness}), we may conclude that
$R_{n,r,s}(p_j)=0$, and thus $$
       \lambda_j = Q_{n,r,s}(p_j) = \int p_j(t) d\lambda(t) > 0 ,
$$ as claimed in (\ref{main1}).

For a proof of (\ref{main2}), consider the polynomial
\begin{equation} \label{main2ter}
     P_j(t) = P_{n,r,s,a,j}(t)= \frac{(t-a)^j}{j!} \Omega_{r-j-1}(t) \omega(t) ,
             \quad
              \omega(t)=(b-t)^s \prod_{k=1}^n (\tau_k-t)^2 ,
\end{equation}
where $\Omega_m$ denotes the $m$th partial sum of the Taylor
expansion of $1/\omega$ at $t=a$. Writing shorter $$
       D^k_c f = \frac{f^{(k)}(c)}{k!} ,
$$ we observe that, by construction, $$
       \forall k=1,...,n : \,\, D^0_{\tau_k} P_j = 0 , \quad
       \forall k=0,...,s-1 : \,\, D^k_b P_j = 0 , \quad
       \forall k=0,...,j-1 : \,\, D^k_a P_j = 0 .
$$ Furthermore, for $k=j,j+1,...,r-1$ we find by the Leibniz
product rule and by definition of $\Omega_m$ that $$
       D^k_a P_j =
       \sum_{\ell=0}^k \Bigl[ D^{\ell}_a \frac{(t-a)^j}{j!} \Bigr] \,
       \Bigl[ D^{k-\ell}_a (\Omega_{r-j-1} \omega) \Bigr]
       = \frac{1}{j!} D^{k-j}_a (\Omega_{r-j-1} \omega)
       = \frac{1}{j!} \delta_{k-j,0} .
$$ Since in addition $P_j \in \mathbb P_{2n+r+s-1}$, we may
conclude that \begin{equation} \label{main2bis}
    \lambda_0^{(j)} = \lambda_0^{(j)} \, j! \, D_a^j P_j =
    Q_{n,r,s}(P_j) = \int P_j(t) \, d\lambda(t) .
\end{equation}
In order to discuss the sign of the expression on the right, we
need the following auxiliary result.

\begin{lemme}
    Let $P(t)=\prod_{j=1}^m (x_j-t)$, with $x_1,...,x_m\in
    (c,+\infty)$, then
    $$
       \forall \ell = 0, 1 ,... : \qquad D_c^\ell \Bigl( \frac{1}{P} \Bigr) > 0
       .
    $$
\end{lemme}
\begin{proof}
   For $m=1$ we find that $D_c^\ell \bigl( \frac{1}{P}
   \bigr)=\frac{1}{(x_1-c)^{\ell+1}}>0$. The general case follows by
   induction on $m$ using the Leibniz product rule.
\end{proof}

As a consequence of the preceding lemma, we find that $$
      \Omega_m(t) = \sum_{\ell=0}^m (t-a)^\ell D_a^\ell \Bigl(\frac{1}{\omega}\Bigr)
$$ is strictly positive on $(a,b)$ for all $m \geq 0$, and thus
$P_j$ defined in (\ref{main2ter}) is also non negative in $(a,b)$.
It follows from (\ref{main2bis}) that $\lambda_0^{(j)}>0$, as
claimed in (\ref{main2}).

Finally, for a proof of (\ref{main3}) we observe that the variable
transformation $t'=-t$ allows to exchange the roles of
$(r,a,\lambda_0^{(j)})$ and $(s,b,\lambda_{n+1}^{(j)})$ in the
quadrature formula (\ref{gauss_radau_lobatto}), and in particular
gives a factor $(-1)^j$ for the $j$th derivative. Hence the
assertion (\ref{main3}) follows from (\ref{main2}), but it is also
straight forward to give a direct proof following the above lines.
\end{proof}

\begin{remark}\label{positivity}
  {\rm Notice that also $D_{\tau_k}^1 P_j=0$ for $k=1,2,...,n$.
  Thus, for polynomial interpolation (in the sense of Hermite)
  at the node $a$ with multiplicity $r$, the nodes $\tau_1,...,\tau_n$
  with multiplicity $2$,
  and $b$ with multiplicity $s$, we have shown implicitly that
  the Lagrange polynomials $P_j=P_{n,r,s,a,j}$ associated with
  the $j$th derivative at $a$ do
  not change sign on $[a,b]$. It follows by symmetry that the
  Lagrange polynomial $P_{n,r,s,b,j}$ associated with the $j$th
  derivative at $b$
  has constant sign $(-1)^j$ on $[a,b]$. However, the Lagrange
  polynomials associated with $\tau_j$ may very well change sign on $(a,b)$.}
\end{remark}

The positivity of the quadrature weights is the essential key for
proving the following convergence result both for generalized
Gauss-Radau and for Gauss-Lobatto formulae.

\begin{corollary}\label{banach_steinhaus}
   Let $[a,b]$ be compact, and $q:=\max\{ r-1,s-1\}$. Then for
   any $f\in \mathcal C^{q}([a,b])$ we have
   $$
         \lim_{n\to \infty} Q_{n,r,s}(f)=\int f(t) \, d\lambda(t).
   $$
\end{corollary}
\begin{proof}
   In the sequal of this proof we suppose that $r,s\geq 1$, the
   extension of the proof for $r=0$ or $s=0$ is straight forward.
   It is not difficult to see that the
   space $X=\mathcal C^{q}([a,b])$ equipped with the norm
   $$
                \| f \| = \max_{0\leq j \leq q} \max_{x\in [a,b]}
                |f(x)|
   $$
   gives a Banach space: the completeness follows immediately from
   the well-known completeness of the space $\mathcal C([a,b])$
   with respect to the maximum norm on $[a,b]$, see,
   e.g., \cite[p.~258]{quef}. Also, from
   \cite[Theorem~6.3.2]{davis} it follows that
   polynomials are dense in $(X,\| \cdot \|)$. 
   Hence we are
   prepared to apply the Banach-Steinhaus Theorem:
   for $f$ being a
   polynomial of degree $k$, we obtain from (\ref{exactness}) that
   $$
          \forall n > \frac{k-r-s}{2} :   \qquad
          Q_{n,r,s}(f)=\int f(t) \, d\lambda(t),
   $$
   i.e., we have convergence for a dense subset of $(X,\|
   \cdot \|)$. 
   For obtaining convergence in $X$ it only remains to show
   that the norm of
   the linear functionals $Q_{n,r,s}$ is bounded uniformly in $n$.
   Writing more explicitly
   $\lambda_0^{(j)}(n),\lambda_{j}(n),\lambda_{n+1}^{(j)}(n)$
   for the quadrature weights occurring in
   (\ref{gauss_radau_lobatto}), we obtain the simple upper bound
   \begin{equation} \label{norm_Q}
      \| Q_{n,r,s}  \| \leq
            \sum_{j=0}^{r-1} \lambda_0^{(j)}(n) +
            \sum_{j=1}^n \lambda_j(n)  +
      \sum_{j=0}^{s-1} \lambda_{n+1}^{(j)}(n) ,
   \end{equation}
   since, according to Theorem~\ref{main}, all weights occurring in
   these sums are positive. We observe that
   $$
            \lambda_0^{(0)}(n) +
            \sum_{j=1}^n \lambda_j(n)  +
      \lambda_{n+1}^{(0)}(n) = Q_{n,r,s}(1)= \int d\lambda(t) <
      \infty.
   $$
   For the remaining terms we consider the polynomial (not depending on
   $n$)
   $$
          P(x):= \sum_{j=1}^{r-1} P_{0,r,s,a,j}(x) +
      \sum_{j=1}^{s-1} (-1)^j P_{0,r,s,b,j}(x)
   $$ of degree $\leq r+s-1$, where we recall from Remark~\ref{positivity}
   that each term in the above sums, representing up to a sign a Lagrange polynomial in
   the sense of Hermite at the abscissa $a$ with multiplicity $r$
   and $b$ with multiplicity $s$, is non negative on $[a,b]$.
   Hence this polynomial $P$ is also non negative
   on $[a,b]$, implying that, again by the positivity of the
   weights,
   $$
       \sum_{j=1}^{r-1} \lambda_0^{(j)}(n) +
      \sum_{j=1}^{s-1} \lambda_{n+1}^{(j)}(n)
      \leq Q_{n,r,s}(P) = \int P(t) \, d\lambda(t)
   $$
   for all $n \geq 0$, where in the last equality we have used
   (\ref{exactness}).
   Hence the expression on the right of (\ref{norm_Q}) is bounded
   uniformly in $n$, and the Banach-Steinhaus Theorem allows us to
   conclude that there is convergence as claimed in the assertion
   of Corollary~\ref{banach_steinhaus}.
\end{proof}

\begin{remark}\label{composite}
  {\rm By using classical arguments we may also estimate the rate
  of convergence of our quadrature formula: according to
  (\ref{exactness}), we find the following bound for the error
  \begin{equation} \label{rate}
      | R_{n,r,s}(f) | \leq \Bigl( \int d\lambda(t) + \| Q_{n,r,s} \|\Bigr)
      \, \inf_{p \in \mathbb P_{2n+r+s-1}} \| f - p \| .
  \end{equation}
  Here we can give a quite rough explicit upper bound for
  $\| Q_{n,r,s} \|$ which is independent of $n$:
  by having a closer look at the construction
  of the polynomials
  $P_{0,r,s,a,j}$ from (\ref{main2ter}) we see that
  $$
    \forall x\in [a,b] :
       \qquad 0 \leq P_{0,r,s,a,j}(x) \leq \frac{(x-a)^j}{j!}
       \sum_{\ell=0}^{r-j-1} \frac{(x-a)^\ell}{(b-a)^{s+\ell}}
       \leq r(b-a)^r ,
  $$ and a similar bound for $(-1)^j P_{0,r,s,b,j}(x)$.
  Consequently, we learn from the previous proof and especially
  from (\ref{norm_Q}) that
  \begin{equation} \label{rate2}
      \| Q_{n,r,s} \| \leq \Bigl(1+ r^2(b-a)^r + s^2(b-a)^s\Bigr) \int
      d\lambda(t).
  \end{equation}
  Using (\ref{rate}) and (\ref{rate2}), it is possible to show
  also convergence for a composite quadrature rule based on
  suitably shifted and scaled counterparts of $Q_{n,r,s}$, and
  to derive an explicit rate of convergence in terms of the size
  of the largest underlying subinterval.}
\end{remark}

\section{Rate of convergence for analytic functions}

Denote by $P_f\in \mathbb P_{2n+r+s-1}$ the polynomial
interpolating $f$ with multiplicity $r$ in $a$, multiplicity $s$
in $b$ and multiplicity $2$ at the other abscissae $\tau_j$
occurring in (\ref{gauss_radau_lobatto}) for $j=1,...,n$, then
using the Cauchy error formula for polynomial interpolation we get
from (\ref{exactness}) that the error for our quadrature formula
for $f\in \mathcal C^{2n+r+s}([a,b])$ may be written in terms of
divided differences as
\begin{eqnarray}
    R_{n,r,s}(f) &=& \nonumber
    \int \Bigl( f(t)-P_f(t)\Bigr) \, d\lambda(t)
    \\&=& \nonumber
    \int  (t-a)^r (t-b)^s \prod_{j=1}^n (t-\tau_j)^2
    [\underbrace{a,...a}_r,\underbrace{\tau_1,\tau_1}_2,...,\underbrace{\tau_n,\tau_n}_2,
    \underbrace{b,...,b}_s,t] f \, d\lambda(t)
    \\&=& \label{divdif}
    [a,...a,\tau_1,\tau_1,...,\tau_n,\tau_n,
    b,...,b,\xi_{n,r,s}] f \int  (t-a)^r (t-b)^s \prod_{j=1}^n (t-\tau_j)^2
    \, d\lambda(t)
    \\&=& \nonumber \frac{f^{(2n+r+s)}(\xi'_{n,r,s})}{(2n+r+s)!} \int  (t-a)^r (t-b)^s \prod_{j=1}^n (t-\tau_j)^2
    \, d\lambda(t)
\end{eqnarray}
with $\xi_{n,r,s},\xi'_{n,r,s} \in [a,b]$, since the polynomial
factor in the integral is of unique sign. Denote by $
   p_{n,r,s}
$ 
the 
orthonormal polynomial with respect to the modified weight
$(t-a)^r(b-t)^s \, d\lambda(t)$. We suppose that $\lambda$ has the
compact support $[-1,1]$ with $a\leq-1<1\leq b$, then it is
well-known from, e.g., \cite[Section 11.11]{He2} that
\begin{equation} \label{r=0,s=0}
     \limsup_{n \to \infty} | R_{n,0,0}(f)|^{1/n} \leq 1/\rho^2 < 1
\end{equation}
provided that $f$ is analytic in the closed ellipse $\mathcal
E_\rho$ with foci $\pm 1$ and half axes $(\rho \pm 1/\rho)/2$, and
that this result is optimal for measures satisfying the Szeg\"o
condition. A similar rate is shown to be true for fixed $r,s>0$,
and we are curious about the rate of convergence if $r=r_n$ and
$s=s_n$ such that $r_n/n \to \alpha\geq 0$,  $s_n/n \to \beta\geq
0$.

We first notice that $\xi_{n,r,s}\in [-1,1]$ in (\ref{divdif}).
Hence, for a (set of) contour(s) $\mathcal C$ encercling once
$[-1,1]$ and $a,b$ and staying in a neighborhood of $[a,b]$ where
$f$ is analytic, we get from the Cauchy formula for divided
differences and from (\ref{divdif}) $$
   R_{n,r_n,s_n}(f) = \frac{1}{2\pi i} \int_{\mathcal C}
      \frac{f(z)}{(z-\xi_{n,r_n,s_n})(z-a)^{r_n} (b-z)^{s_n} p_{n,r_n,s_n}(z)^2}
      dz,
$$ and thus \begin{equation} \label{rate_varying}
      \limsup_{n \to \infty} | R_{n,r_n,s_n}(f)|^{1/n} \leq
      \limsup_{n\to \infty} \max_{z\in \mathcal C}
      \frac{1}{|z-a|^\alpha |z-b|^\beta |p_{n,r_n,s_n}(z)|^{2/n}} .
\end{equation}
Thus we are left with the question of $n$th roots asymptotics for
orthogonal polynomials with varying weights, which has been the
subject of a number of publications over the last twenty years,
see, e.g., the monograph \cite[Chapters~III.6 and VII]{sato} of
Saff and Totik or the monograph \cite[Chapter~3]{stto} of Stahl
and Totik. Since the negative logarithm of the absolute value of a
polynomial is a logarithmic potential of some discrete measure,
here the right tool to describe the $n$th root asymptotic is to
consider a weighted equilibrium problem in logarithmic potential
theory: the potential and the energy of a Borel measure $\mu$ with
compact support are defined by $$
    U^\mu(y)=\int \log(\frac{1}{|x-y|}) \, d\mu(x), \qquad
    I(\mu) =
    \int \int \log(\frac{1}{|x-y|}) \, d\mu(x)\, d\mu(y).
$$ Define the external field $Q(x)=U^\sigma(x)$, $\sigma=
\frac{\alpha}{2}\delta_a+\frac{\beta}{2}\delta_b$, with $\delta_c$
the Dirac unit point measure at $x=c$, then there exists a unique
probability measure $\mu$ supported on $\Sigma=[-1,1]$ which under
all such measures has minimal weighted energy $I(\mu)+2 \int Q \,
d\mu$, see \cite[Theorem~I.1.3]{sato}. By the same Theorem (see
also \cite[Theorem~I.5.1]{sato}) we also have the equilibrium
conditions that $U^\mu(x)+Q(x)$ is equal to some constant $F$ on
the support of $\mu$, and $\geq F$ in $[-1,1]$. Then according to
\cite{gora} (see also \cite{stto})
\begin{equation} \label{n_root asymptotics}
      \Bigl( \limsup_{n\to \infty} \max_{z\in \mathcal C}
      \frac{1}{|z-a|^\alpha |z-b|^\beta |p_{n,r_n,s_n}(z)|^{2/n}} \Bigr)
      \leq \exp\Bigl( 2 \, \sup_{z\in \mathcal C} \Bigl(U^\mu(z) + Q(z)- F \Bigr) \Bigr),
\end{equation}
with equality iff the starting measure of orthogonality $\lambda$
is sufficiently regular.

For our external field, the extremal measure may be found
explicitly: since $Q$ is convex on $[-1,1]$, it follows from
\cite[Theorem~IV.1.11]{sato} that the support of $\mu$ is an
interval of the form $[A,B]$, with $-1\leq A < B \leq 1$, and also
$a<A$ and $B<b$ since $Q$ becomes $+\infty$ at $a$, $b$ and thus
in a neighborhood of these points we may not have equality in the
equilibrium condition.

\begin{lemme}\label{representation}
   Denoting by $z \mapsto \phi(z)=(2z-A-B+2\sqrt{(z-A)(z-B)})/(B-A)$
   the conformal Riemann map sending the
   exterior of $[A,B]$ onto the exterior of the closed unit disk,
   then  we
   have for $z\in \mathbb C$
   \begin{equation} \label{potential}
           \exp(2 (U^\mu(z) + Q(z)- F)) = \frac{1}{|\phi(z)|^2}
           \Bigl| \frac{1-\phi(z) \phi(a)}{\phi(z)(\phi(z)-\phi(a))}\Bigr|^\alpha
           \Bigl| \frac{1-\phi(z) \phi(b)}{\phi(z)(\phi(z)-\phi(b))}\Bigr|^\beta
 ,
   \end{equation}
   where $-1\leq A < B \leq 1$ are defined by the system
   \begin{eqnarray*} &&
         \frac{\alpha}{2} ( \sqrt{\frac{A-a}{B-a}} - 1 )
         + \frac{\beta}{2} ( \sqrt{\frac{b-A}{b-B}} - 1 )
         \left\{\begin{array}{ll}
            = 1 & \mbox{if $B<1$,}
            \\
            \leq 1 & \mbox{if $B=1$,}
         \end{array}\right.
         \\ &&
         \frac{\alpha}{2} ( \sqrt{\frac{B-a}{A-a}} - 1 )
         + \frac{\beta}{2} ( \sqrt{\frac{b-B}{b-A}} - 1 )
         \left\{\begin{array}{ll}
            = 1 & \mbox{if $A>-1$,}
            \\
            \leq 1 & \mbox{if $A=-1$.}
         \end{array}\right.
   \end{eqnarray*}
\end{lemme}
\begin{proof}
   Denote by $\widehat \sigma$ the balayage measure of $\sigma$
   onto $[A,B]$, then by \cite[Theorem~II.4.4]{sato}, $\mu +
   \widehat \sigma$ is a positive measure of total mass
   $(1+\frac{\alpha}{2}+\frac{\beta}{2})$
   supported on $[A,B]$, and from the
   equilibrium conditions we know that its potential is constant
   quasi everywhere on $[A,B]$. However, by, e.g.,
   \cite[Theorem~I.1.3]{sato}, the only measure satisfying this
   relation is $(1+\frac{\alpha}{2}+\frac{\beta}{2}) \omega_{[A,B]}$,
   with $\omega_{[A,B]}$ the Robin measure of $[A,B]$, i.e.,
   the equilibrium
   measure with external field $0$ on $[A,B]$.
   Hence from \cite[Eqns.~(II.4.32) and~(II.4.35)]{sato} and the fact that
   $U^\mu + Q- F$ equals zero on $[A,B]$ we may conclude that
   $$
           U^\mu(z) + Q(z)- F = \frac{\alpha}{2} g_{[A,B]}(z,a)
           + \frac{\beta}{2} g_{[A,B]}(z,b) -
           (1+\frac{\alpha}{2}+\frac{\beta}{2})
           g_{[A,B]}(z,\infty) ,
   $$
   where by $x \mapsto g_{[A,B]}(x,y)$ we denote the Green function
   of the domain
   $\mathbb C\setminus [A,B]$ with pole at $y \not\in [A,B]$.
   Taking into account the link \cite[Eqn.~(II.4.45)]{sato}
   between the Green function and the Riemann map, relation
   (\ref{potential}) follows.
   Finally, the so-called F-functional of \cite[Theorem~IV.1.5]{sato}
   \begin{eqnarray*}
       F(A,B)&=&\log(\frac{B-A}{4}) - \frac{1}{\pi}\int_A^B \frac{Q(x)\, dx}{\sqrt{(x-A)(B-x)}}
       \\&=&
       (1+\frac{\alpha}{2}+\frac{\beta}{2})\log(\frac{B-A}{4})
       + \frac{\alpha}{2} g_{[A,B]}(a,\infty)
       + \frac{\beta}{2} g_{[A,B]}(b,\infty)
   \end{eqnarray*}
   must take its global maximum on $-1 \leq A < B \leq 1$
   at the endpoints of the support
   of the extremal measure $\mu$. Taking partial derivatives,
   we arrive at the given system of
   equations and inequalities for $A$, $B$ as in
   \cite[Theorem~II.4.4 and Lemma~II.1.15]{sato}, compare
   with \cite[Example~II.1.17]{sato} for the special case $a=-1$ and $b=1$
   of Jacobi weights.
\end{proof}

\begin{figure}
  \centerline{\epsfig{file=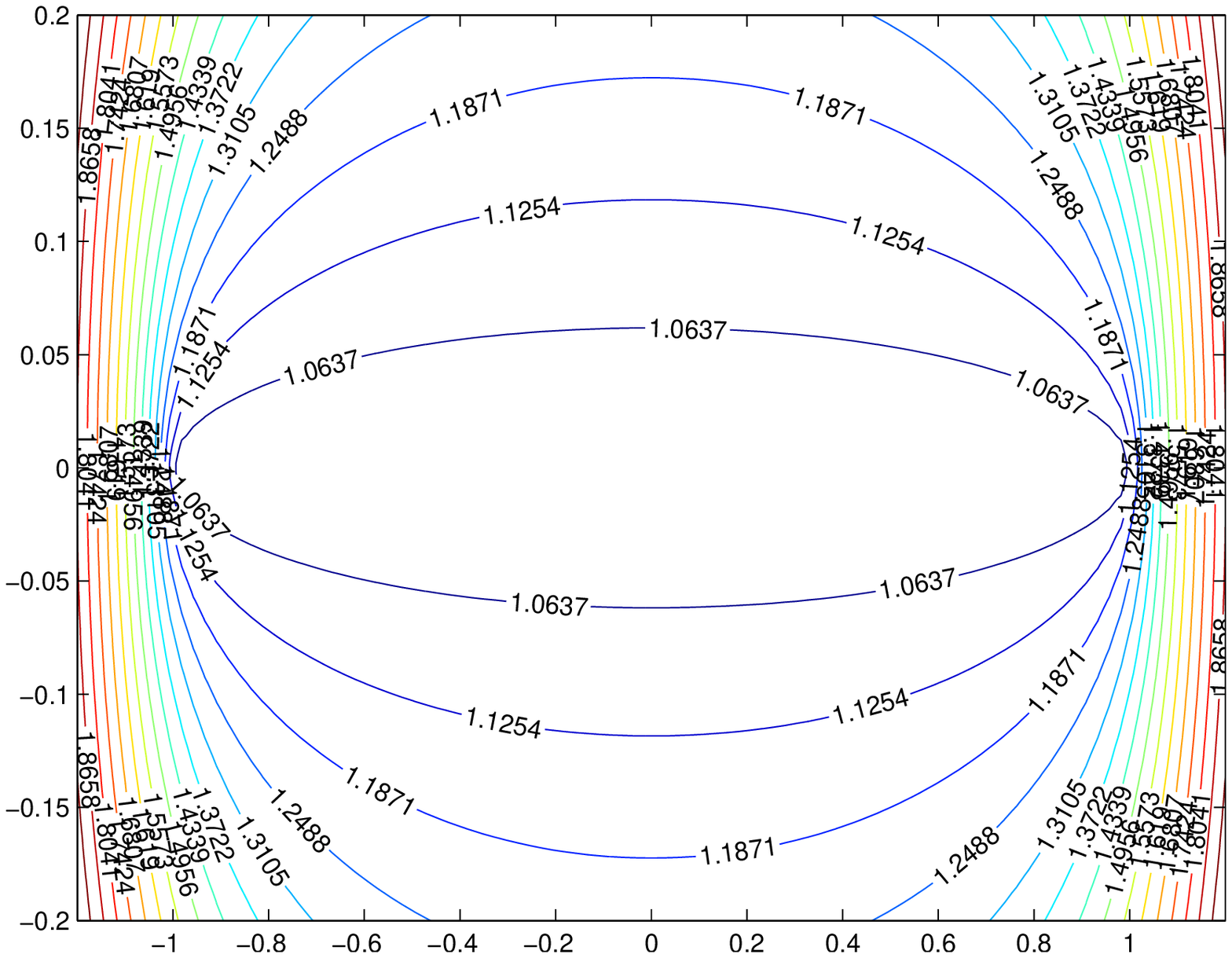, scale=0.5}~~~
              \epsfig{file=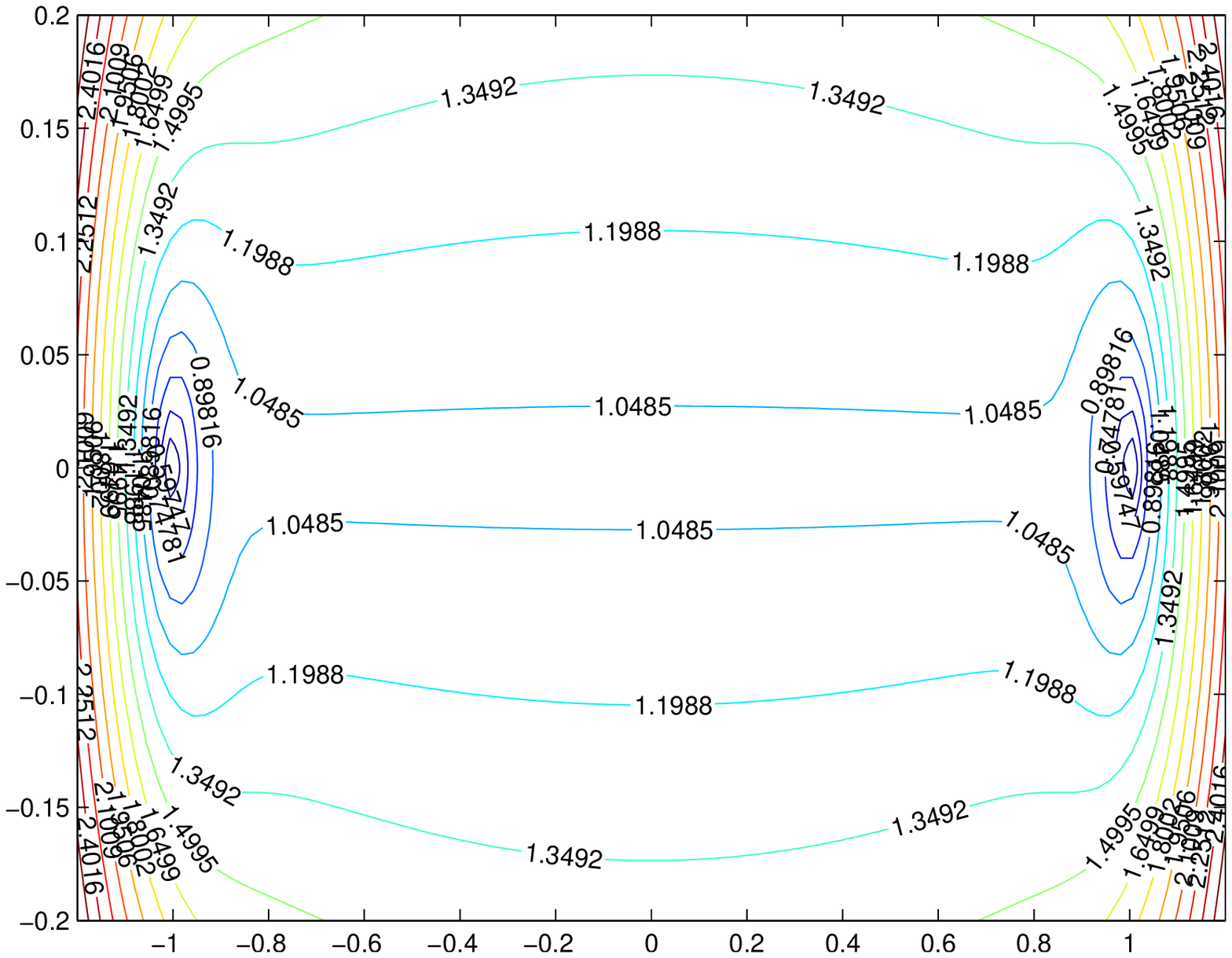, scale=0.5}}
  \caption{{\it Level curves for different choices of the parameters
  $a,\alpha,b,\beta$: on the left we find the classical case $\alpha=\beta=0$,
  here $[A,B]=[-1,1]$ and the level lines are ellipses. On the right
  $\alpha=\beta=1$ and $b=-a=1$, the endpoints of the support of orthogonality,
  here the level sets for $\rho>1$ are connected with connected
  complement, in this case $B=-A\approx 0.86603$.}}\label{figure1}
\end{figure}

\begin{figure}
  \centerline{\epsfig{file=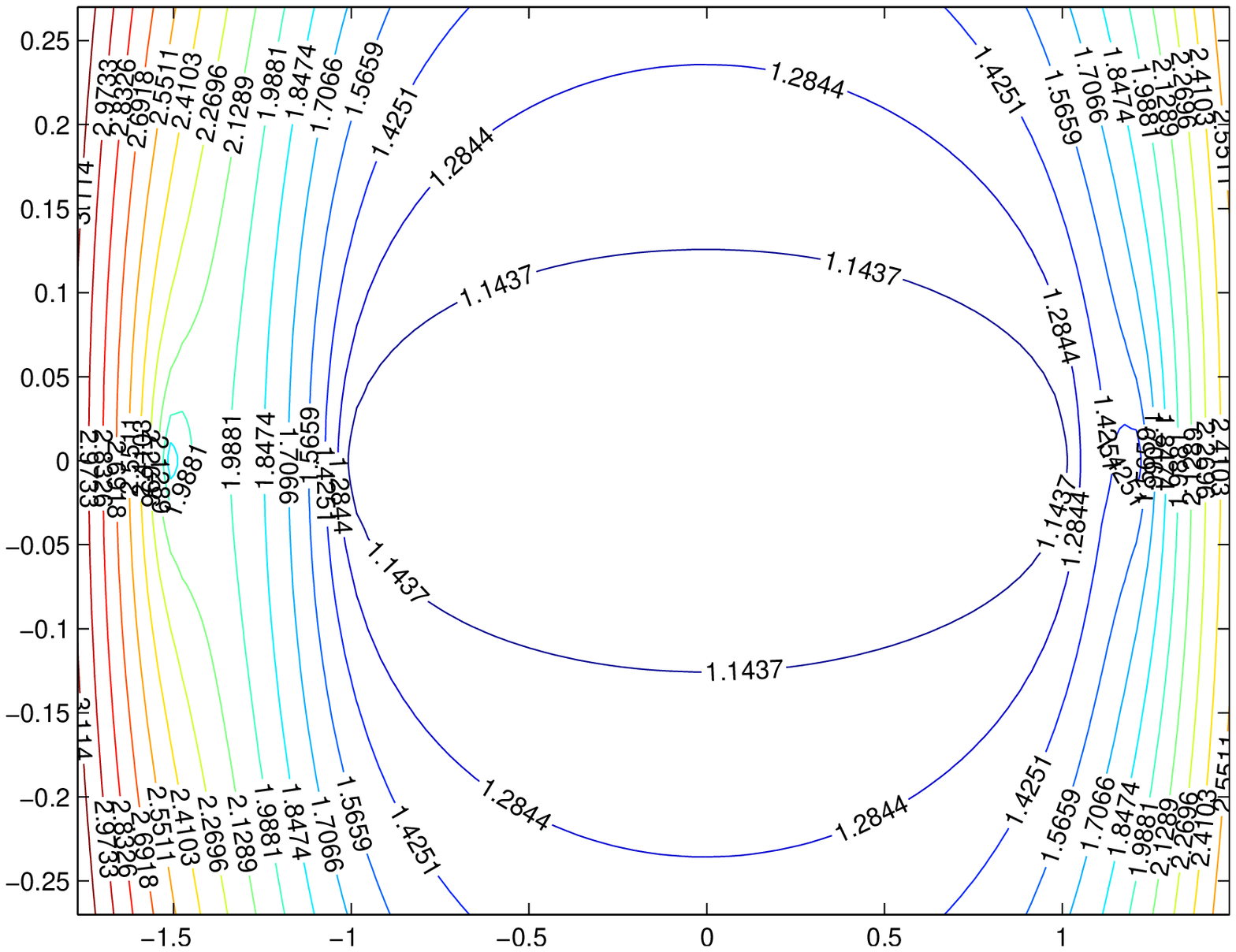, scale=0.5}~~~
              \epsfig{file=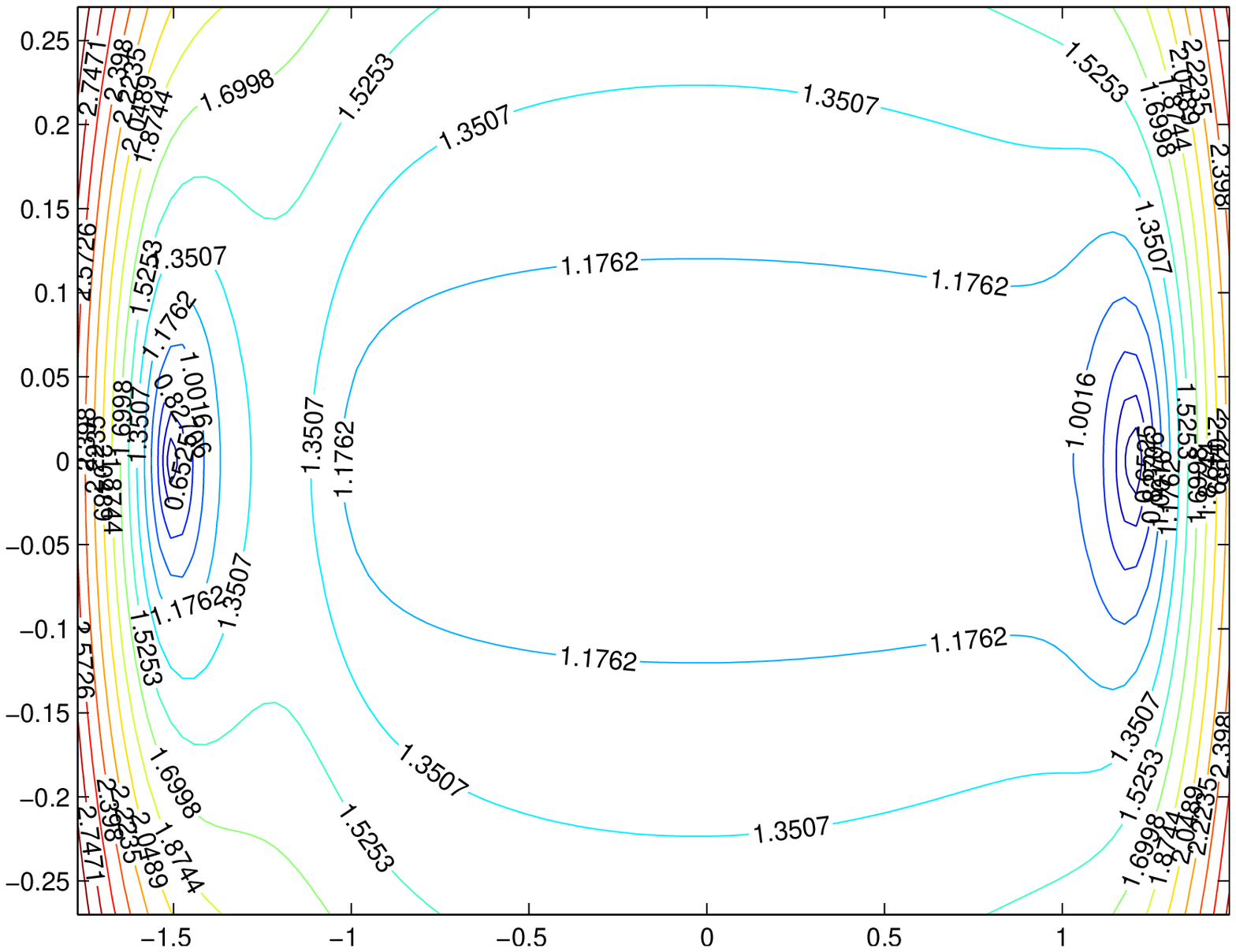, scale=0.5}}
  \caption{{\it Level curves for different choices of the parameters
  $\alpha=\beta$: in both cases, $a=-1.5,b=1$, but
  on the left we find $\alpha=\beta=0.2$, on the right
  $\alpha=\beta=1$, leading to $[A,B]=[-1,1]$ in both cases.
  In particular, for small $\rho>1$ we find three connected
  components for our level sets.}}\label{figure2}
\end{figure}

\begin{figure}
  \centerline{\epsfig{file=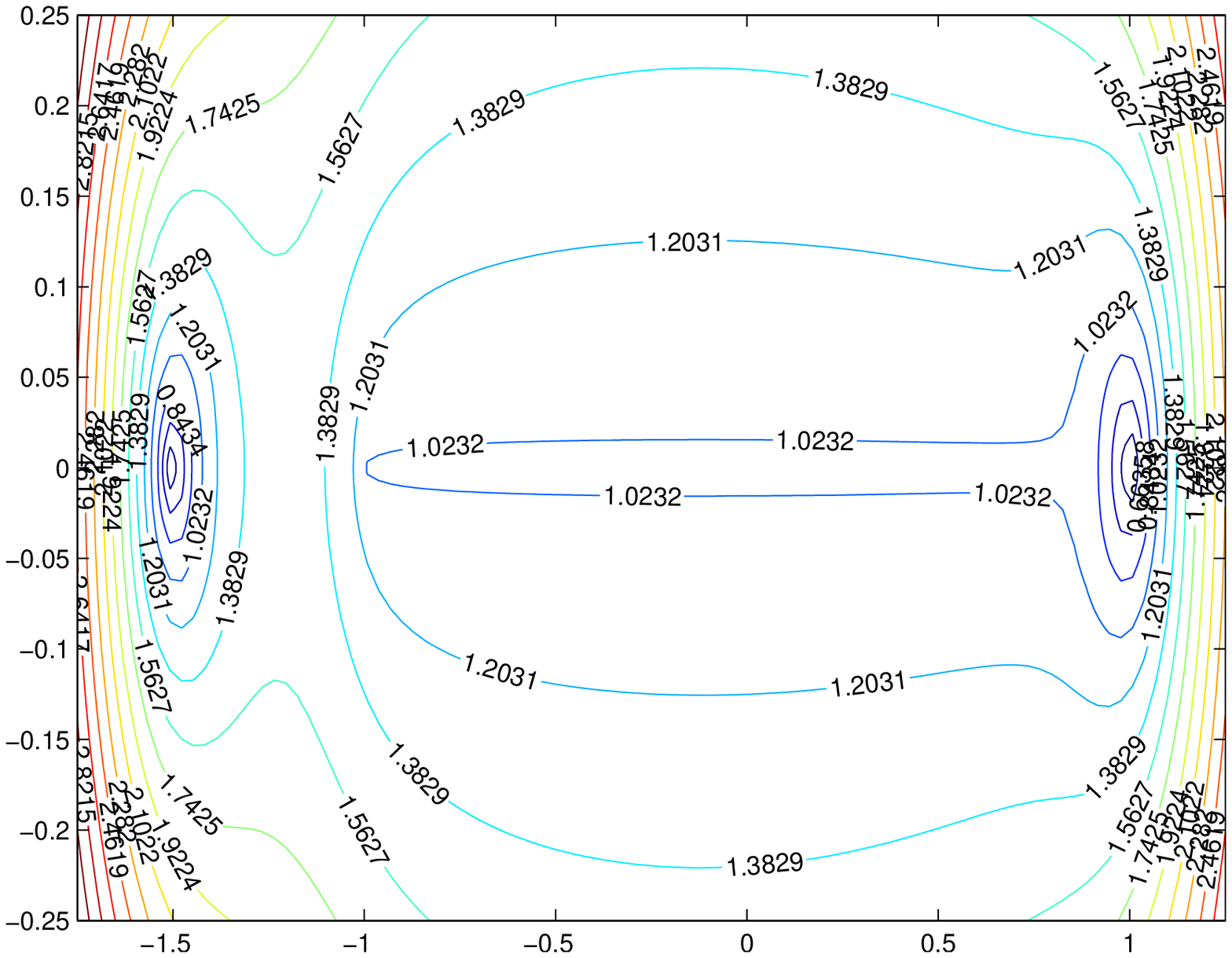, scale=0.5}~~~
              \epsfig{file=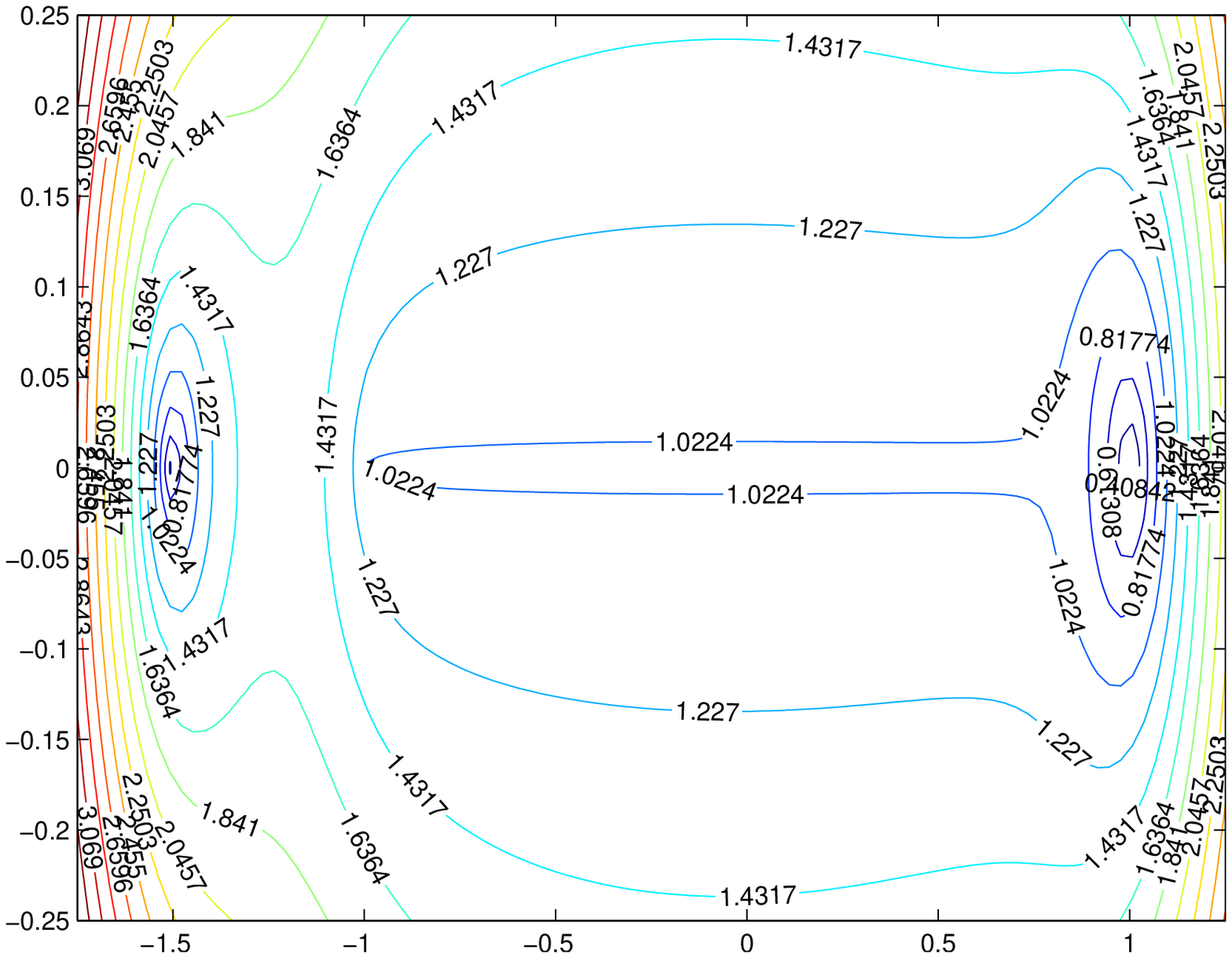, scale=0.5}}
  \caption{{\it Level curves for different choices of the parameters
  $\beta$: in both cases, $a=-1.5,b=1$, $\alpha=1$, but
  $\beta =1$ on the left (leading to $[A,B]=[-1,0.8402]$),
  and $\beta=1.2$ on the right (leading to $[A,B]=[-1,0.79334]$).
  Here for small $\rho>1$ we find two connected
  components for our level sets.}}\label{figure3}
\end{figure}

In order to exploit Lemma~\ref{representation}, we have to
consider for $\rho>1$ the (closed) level sets $\mathcal
E_\rho(a,\alpha,b,\beta)$ being the complement of the set of $z\in
\mathbb C \setminus[A,B]$ where the right-hand side of
(\ref{potential}) is $<
\rho^{-2}<1$. Notice that for $\alpha=\beta=0$ we have
$[A,B]=[-1,1]$, and we obtain for the complement the reqirement
$|\phi(z)|>\rho$, that is, $\mathcal E_\rho(a,0,b,0)$ coincides
with the ellipses $\mathcal E_\rho$ considered before.  Also, by
the equilibrium conditions, $[-1,1] \cup \{a,b\} \subset \mathcal
E_\rho(a,\alpha,b,\beta)$ for all $\rho
> 1$, and from the maximum principle for analytic functions we may
conclude that $\mathcal E_\rho(a,\alpha,b,\beta)$ has a connected
complement containing a neighborhood of infinity, and at most
three connected components, one of them containing $a$ (if
$\alpha>0$), a second $b$ (if $\beta>0$), and the third the
interval $[-1,1]$, see Figure~\ref{figure2}. Moreover, if $A>-1$
(and similarly $B<1$), then $\mu$ is also an extremal measure if
we replace $\Sigma=[-1,1]$ by the larger set $[a,1]$, and hence
$[a,1]\subset \mathcal E_\rho(a,\alpha,b,\beta)$, showing that
there are only at most two connected components (see
Figure~\ref{figure3}).

If we choose as $\mathcal C$ the boundary of some level set
$\mathcal E_\rho(a,\alpha,b,\beta)$ for some $\rho > 1$, this (set
of) contour(s) encircles once $a,b$, and the interval $[-1,1]$. A
combination of (\ref{rate_varying}), (\ref{n_root asymptotics}),
and (\ref{potential}) leads to the following result.

\begin{theorem}\label{analytic}
   Suppose that $r_n/n \to \alpha \geq 0$, $s_n/n \to \beta \geq 0$,
   and let $f$ be analytic in
   $\mathcal E_\rho(a,\alpha,b,\beta)$ for some $\rho > 1$, then
   $$
      \limsup_{n \to \infty} | R_{n,r_n,s_n}(f)|^{1/n} \leq
      \rho^{-2}.
   $$
\end{theorem}

   One may show that again this estimate is best possible if
   the orthogonality measure $\lambda$ is sufficiently regular.
   In addition, if $\lambda$ satisfies the Szeg\H{o} condition,
   then following \cite{totik} we may obtain strong
   asymptotics for the orthonormal polynomials $p_{n,r_n,s_n}$,
   and hence with help of steepest descend an asymptotic
   equivalent of $| R_{n,r_n,s_n}(f)|$.

Different examples for the level sets of the preceding theorem are
given in Figure~\ref{figure1}, Figure~\ref{figure2}, and
Figure~\ref{figure3}. Though the shapes of these sets are quite
different depending on the parameters, there seem to be clearly an
indication: if the function $f$ is regular in larger neighborhoods
around $a$ and $b$, but not in such a large neighborhood around
for instance $0$ (which is true for instance for the function
$f(z)=1/(1+z^2)$), then by the choice of larger $\alpha,\beta$ one
improves the rate of geometric convergence of our generalized
Gauss-Lobatto quadrature formula.

\end{document}